\documentclass[11pt]{article}
\usepackage{amsmath,amssymb,a4wide,amsthm,color,graphicx,verbatim,algorithmic,algorithm,caption,enumerate,multirow}
\newtheorem{theorem}{Theorem}[section]

\newtheorem{remark}[theorem]{Remark}
\newtheorem{lemma}[theorem]{Lemma}
\newtheorem{proposition}[theorem]{Proposition}
\newtheorem{ex}[theorem]{Example}

\newtheorem{definition}[theorem]{Definition}

\newcommand{\norm}  [1]{\ensuremath{\left  \|       #1  \right \|       }}

\newcommand{\cb}    [1]{\ensuremath{\left  \{      #1  \right \}       }}

\newcommand{\st} {\ensuremath{|\;}}

\newcommand{\epi} {{\rm epi \,}}

\newcommand{\cl}  {{\rm cl  \,}}
\newcommand{\bd}  {{\rm bd \,}}

\newcommand{\Int} {{\rm int \,}}
\newcommand{\conv}  {{\rm conv \,}}
\newcommand{\cone}{{\rm cone\,}}

\newcommand{\R}{\mathbb{R}}
\newcommand{\N}{\mathbb{N}}
\newcommand{\smz}{\setminus\{0\}}

\newcommand{\Minc}{{\rm Min}_C\,}
\newcommand{\wMinc}{{\rm wMin}_C\,}
\newcommand{\wMinK}{{\rm wMin}_K\,}

\newcommand{\MaxK}{{\rm Max}_K\,}

\newcommand{\recc}{{\rm recc \,}}

\author{Firdevs Ulus  
}

\title{Tractability of Convex Vector Optimization Problems in the Sense of Polyhedral Approximations}

\date{\today}

\begin{document}
\maketitle

\begin{abstract} \noindent
There are different solution concepts for convex vector optimization problems (CVOPs) and a recent one, which is motivated from a set optimization point of view, consists of finitely many efficient solutions that generate polyhedral inner and outer approximations to the Pareto frontier. A CVOP with compact feasible region is known to be bounded and there exists a solution of this sense to it. However, it is not known if it is possible to generate polyhedral inner and outer approximations to the Pareto frontier of a CVOP if the feasible region is not compact. This study shows that not all CVOPs are tractable in that sense and gives a characterization of tractable problems in terms of the well known weighted sum scalarization problems. 
\medskip

\noindent
{\bf Keywords:} Vector optimization, multiobjective optimization, convex programming, polyhedral approximation.

\medskip

\noindent
{\bf MSC 2010 Classification:} 90C29, 90C25

\end{abstract}

\section{Introduction} \label{sect:intro}
Vector optimization problem with a finite dimensional image space is to minimize an $\R^q$-valued objective function with respect to the partial order induced by an ordering cone $K \subseteq \R^q$ over a feasible region. Whenever the ordering cone is the positive orthant, the problem is called a multi-objective optimization problem, namely $q$ objective functions are to be minimized with respect to the component-wise ordering. 

There are different solution concepts regarding vector optimization problems. A \emph{minimizer} (efficient solution or Pareto optimal solution) for instance, is a singleton in the feasible set which is not dominated by the other feasible points. Similarly, a \emph{weak minimizer} (weakly efficient solution or weak Pareto optimal solution) is a feasible solution which is not strictly dominated. Note that a solution $x$ is said to (strictly) dominate another solution $\bar{x}$ if the image of $x$ is (strictly) less than the image of $\bar{x}$ with respect to ordering cone $K$.

More recently, a solution concept, which is motivated from a set optimization point of view, has been introduced for vector optimization problems by Heyde and L\"ohne in~\cite{freshlook}. Accordingly, a solution consists of minimizers which, together with the extreme directions of the ordering cone, generate the Pareto frontier. Clearly, for linear vector optimization problems (LVOPs) a solution in this sense contains finitely many minimizers and there are algorithms to generate one, see, for instance,~\cite{benson,ehr_dual,ehrgott_simplex,steuer73,lvop,geometric_lp,lvopalg}. 

In~\cite{cvop}, an \emph{$\epsilon$-solution} concept is given as a finite set of weak minimizers which generates an inner and an outer approximation to the Pareto frontier. There are Benson-type algorithms that find \emph{$\epsilon$-solutions} to convex vector optimization problems (CVOPs) assuming that the feasible region is compact, see~\cite{ehrgott, cvop}. Note that if the feasible region is compact, then the problem is \emph{bounded}, that is, the image of the feasible region in the objective space is included in a shifted cone, namely ${a}+K$ for some $a \in \R^q$ and $K$ being the ordering cone of the problem. In some applications, the feasible region of a convex vector optimization problem is not compact. Computation of some set-valued risk measures~\cite{cvop} can be given as an example of such cases. In general, the problem of interest may be unbounded or it may be difficult to check if it is a bounded problem. 

The aim of this study is to understand the structure of possibly unbounded CVOPs and to see if these problems are \emph{tractable} in the sense that there exist polyhedral outer and inner approximations to the Pareto frontier. We provide a simple example for which the problem is not bounded and it is not possible to find polyhedral outer and inner approximations such that the Hausdorff distance between the two is finite. On the other hand, the existence of unbounded, but \emph{tractable} problems is known. For instance, it is possible to generate a solution to a linear vector optimization problem as long as the Pareto frontier exists. 

Here, we provide a characterization of \emph{tractable} CVOPs depending on the recession cone of the \emph{upper image} (image of the feasible region added to the original ordering cone) of the problem. Accordingly, there exists a polyhedral inner and outer approximations to the Pareto frontier of a CVOP if and only if the problem is bounded with respect to the ordering cone taken as the recession cone of the {upper image} of the problem. We call such problems \emph{self-bounded}.

We give a characterization of the recession cone of the upper image of a {self-bounded} problem in terms of the well-known weighted sum scalarization problems. Accordingly, for a self-bounded problem, the set of weights which makes the weighted sum scalarization problem bounded is equal to the (positive) dual cone of the recession cone. Moreover, we also show the reverse implication, that is, if these two sets are equal, then the problem is self-bounded.

This paper is structured as follows. Section 2 is dedicated to basic concepts and notation. In Section 3, some results on convex upper closed sets are provided. The convex vector optimization problem, its solution concepts and the main results of the paper are provided in Section 4. Section 5 provides some concluding remarks.

\section{Preliminaries}
\label{sect:Prelim}

A subset $K$ of $\R^q$ is a \emph{cone} if $\lambda k \in K$ when $k\in K$ and $\lambda >0$. For a set $A \subseteq \R^q$, the interior, closure, boundary, convex hull and the conic hull of $A$ are denoted respectively by $\Int A$, $\cl A$, $\bd A$, $\conv A$, and $\cone A$. Moreover, $k \in \R^q\setminus\{0\}$ is called a direction of $A$ if $\{a+\alpha k \in \R^q \st a\in A, \alpha > 0 \} \subseteq A$. The \emph{recession cone} $\recc A$ of $A$ consists of the directions of $A$, that is, 
\begin{align}\label{eq:recc}
\recc A = \{k\in \R^q \st \forall a\in A, \forall \alpha \geq 0: \: a +\alpha k \in A \}.
\end{align}

A polyhedral convex set $A \subseteq \R^q$ can be written as
\begin{align}\label{Vrep}
A = \conv\{x^1,\ldots, x^s\}+\conv \cone \{k^1,\ldots,k^t\},
\end{align}
where $s \in \N \setminus \{0\}, t \in \N $, each $x^i \in \R^q$ is a point, and each $k^j \in \R^q \setminus \{0\}$ is a {direction} of $A$. The set of points $\{ x^1,\ldots, x^s \}$ together with the set of directions $\{k^1, \ldots , k^t \}$ are said \emph{to generate} the polyhedral convex set $A$. 
Throughout, we also consider (not necessarily polyhedral) convex subsets of $\R^q$ in the form $A = \conv\{x^1,\ldots,x^s\} + K$ where $K \subseteq \R^q$ is a convex cone. In this case, the set of points $\{ x^1,\ldots, x^s \}$ together with cone $K$ are said \emph{to generate} $A$.

The \emph{distance} from a point $y \in \R^q$ to a set $A \subseteq \R^q$ is given by $d(y,A):=\inf_{a \in A}\norm{y-a}$. The \emph{Hausdorff distance} between two closed sets $A_1$ and $A_2$ is given by 
\begin{equation}\label{eq:Hausdorff_dist}
h(A_1,A_2) = \max\{\sup_{a_1\in A_1}d(a_1, A_2),\sup_{a_2\in A_2}d(a_2, A_1)\}.
\end{equation}

A convex cone $C$ is said to be \emph{solid}, if it has a non-empty interior; \emph{pointed} if it does not contain any line through 0; and \emph{non-trivial} if $\emptyset \neq C \neq \R^q$. A non-trivial convex pointed cone C defines a partial ordering $\leq_C$ on $\R^q$: $v \leq_C w$ if and only if $w - v \in C$. Let $C\subseteq \R^q$ be a non-trivial convex pointed cone and $X \subseteq \R^n$ a convex set. A function $f:X \rightarrow \R^q$ is said to be \emph{$C$-convex} if $f(\alpha x+(1-\alpha)y) \leq_C \alpha f(x)+(1-\alpha)f(y)$ holds for all $x,y \in X$, $\alpha \in [0,1]$, see e.g., \cite[Definition 6.1]{luc}. 

For a pointed cone $C$, a point $y \in A$ is called \emph{$C$-minimal element} of A if $\left( \{y\} -  C \setminus \{0\}\right) \cap A = \emptyset$. If cone $C$ is also solid, then a point $y \in A$ is called \emph{weakly $C$-minimal element} if $ \left( \{y\} - \Int C \right) \cap A = \emptyset$. The set of all $C$-minimal elements of $A$ and weakly $C$-minimal elements of $A$ are denoted by $\Minc(A)$ and $\wMinc(A)$, respectively. The (positive) dual cone of $C$ is the set $C^+:=\cb{z \in \R^q\st \forall y \in C: z^T y \geq 0}$.

Let cone $C\subseteq R^q $ be nontrivial and convex. A set $A \subseteq \R^q$ is said to be \emph{upper closed with respect to $C$} if $A = \cl (A+C)$, and \emph{convex upper closed with respect to $C$} if $A = \cl \conv (A+C)$. The collection of such sets are denoted by $\mathcal{G}(\R^q,C)$, that is, 
\begin{equation}\label{eq:G(Rq,C)}
\mathcal{G}(\R^q,C):= \{A \subseteq \R^q \st A = \cl \conv (A +C)\}.
\end{equation} 

\begin{remark}\label{rem:completelattice}
It is known that $(\mathcal{G}(\R^q, C), \oplus, \odot)$ is a partially ordered conlinear space with neutral element $\cl C$, where $B_1 \oplus B_2 := \cl (B_1+B_2)$ and $\alpha\odot B:= \cl (\alpha \cdot B+C)$. Here, $+$ and $\cdot$ are the usual Minkowski summation and multiplication with the conventions $\emptyset+A = A +\emptyset = \emptyset$ for all $A \subseteq \R^q$. Moreover, $\mathcal{G}(\R^q,C)$ is a complete lattice under $\supseteq$ with $$\inf \mathcal{A}= \cl \conv \bigcup_{A \in \mathcal{A}} A,\:\:\: \sup \mathcal{A}= \bigcap_{A \in \mathcal{A}} A,$$ for a nonempty collection $\mathcal{A} \subseteq \mathcal{G}(\R^q,C)$, see, for instance,~\cite{setopt_intro}.
\end{remark}
$A \in \mathcal{G}(\R^q,C)$ is said to be \emph{bounded} if there exists some $y \in \R^q$ with $y + C \supseteq A$. Similarly, for any cone $K\subseteq \R^q$, we say that $A$ is \emph{bounded with respect to $K$} if there exists some $y \in \R^q$ with $y + K \supseteq A$.

Throughout, $B(a,r)$ denotes the closed ball around $a \in\R^q$ with radius $r>0$, that is $B(a,r) = \{y\in \R^q \st \norm{y-a}\leq r\}$, where $\norm{\cdot}$ is the Euclidean norm. The positive orthant in $\R^q$ is $\R^q_+:=\{y\in\R^q \st y_i \geq 0, i=1,\ldots,q\}$.

\section{On Convex Upper Closed Sets} \label{sect:uppersets}

For solving convex vector optimization problems, convex upper closed sets with respect to the ordering cone of the problem play an important role as the image of the set of all weak minimizers can be seen as (a subset of) the boundary of a convex upper closed set known as the \emph{upper image}. Indeed, there are solution concepts for convex vector optimization problems that involve generating (approximations to) the upper image, see, for instance,~\cite{lohne,cvop}. For linear vector optimization problems, it is possible to generate this set by a finite set of points and a finite set of directions~\cite{lohne}, whereas for nonlinear convex vector optimization problems, a solution usually generates an inner and an outer approximation to the upper image~\cite{cvop}. 

If an upper closed set $A$ is known to be bounded with respect to $K$, then it is possible to find an outer approximation to $A$ which is generated by a finite set $\bar{Y} \subset \R^q$ and cone $K$ in the sense that $\conv{\bar{Y}}+K \supseteq A$. If one also wants to generate an inner approximation using the same cone $K$, then $A +K \subseteq A$ needs to be satisfied. The following proposition shows that such cone $K$ needs to be equal to the recession cone $\recc A$ of $A$.

\begin{proposition}\label{prop:upperset1}
Let $A \in \mathcal{G}(\R^q,C)$ be bounded with respect to $K$ for some closed convex cone $K$ which also satisfies $A+K\subseteq A$. Then, $K = \recc A$. 
\end{proposition}
\begin{proof}
As $A$ is bounded with respect to $K$, there exists $y \in \R^q$ such that $y + K \supseteq A$. Then, $K \supseteq \recc A$. On the other hand, $A+K\subseteq A$ implies that $K \subseteq \recc A$.
\end{proof}
\begin{definition} \label{defn:bdd}
A nonempty set $A \in \mathcal{G}(\R^q,C)$ is said to be \emph{self-bounded} if $A\neq \R^q$ and it is bounded with respect to its own recession cone $\recc A$.
\end{definition}
\begin{remark}\label{rem:notsolid}
{Note that the definition of self-boundedness can be extended to sets with recession cones that are not solid. An example of a not self-bounded set would be $A = \epi f(x) \subseteq \R^2$ for $f(x) = x^2$, where $\epi f$ is the epigraph of $f$. Clearly, $\recc A = \{[0, k]^T\st k \geq 0\}$ and there exists no $y \in \R^2$ such that $A \subseteq y + \recc A$.} 		
\end{remark}

Let $\mathcal{B}(\R^q,C)$ be the set of all self-bounded sets together with the whole space, that is, $$\mathcal{B}(\R^q,C):= \{B \in \mathcal{G}(\R^q,C) \st B\text{~is self-bounded or~} B =\R^q\}.$$
Next, we show that $\mathcal{B}(\R^q,C)$ is a conlinear space, however it is not necessarily a complete lattice.
\begin{proposition} \label{prop:conlinearspace}
	$(\mathcal{B}(\R^q,C),\oplus,\odot)$ is a conlinear space with the neutral element $\cl C$.
\end{proposition} 
\begin{proof}
By Remark~\ref{rem:completelattice}, it is enough to show that $B_1 \oplus B_2$ and $\alpha \odot B$ are in $\mathcal{B}(\R^q,C)$ for $B_1, B_2 \in \mathcal{B}(\R^q,C)$, $\alpha \geq 0$,  
 
Let $b \in B_1 \oplus B_2$ for $B_i \in \mathcal{B}(\R^q,C)$, let $y_i \in \R^q$ be such that $B_i \subseteq y_i + \recc B_i$, for $i = 1,2$. Clearly, $b \in y_1+y_2+\cl(\recc B_1 + \recc B_2)$. Note that $B_1 \oplus B_2$ is self-bounded as $\recc B_1 \oplus \recc B_2 \subseteq \recc(B_1\oplus B_2)$ holds. To see the last inclusion, let $r \in \cl(\recc B_1 + \recc B_2)$, that is, $r = \lim_{n \rightarrow \infty} (r_1^{(n)}+r_2^{(n)})$ for some $(r_i^{(n)})_n \in \recc B_i$; let $b\in \cl(B_1+B_2)$, that is, $b = \lim_{n \rightarrow \infty} (b_1^{(n)}+b_2^{(n)})$ for some $(b_i^{(n)})_n \in B_i$. For any $\gamma \geq 0$, we have $b + \gamma r = \lim_{n \rightarrow \infty}(b^{(n)}_1+\gamma r^{(n)}_1 +b^{(n)}_2+\gamma r^{(n)}_2) \in \cl (B_1+B_2)$. 

Let $b \in \alpha \odot B$ for $\alpha \geq 0$, $B \in \mathcal{B}(\R^q,C)$. Let $y\in \R^q$ be such that $B \subseteq y+\recc B$. Note that $b = \lim_{n \rightarrow \infty} (\alpha y+r_n+c_n)$ for some $r_n \in \recc B, c_n \in C$. Then, $b \in \alpha y + \alpha \odot \recc B$. Note that $\alpha \odot B$ is self-bounded as $\alpha \odot \recc B \subseteq \recc(\alpha \odot B)$ holds. To see the last inclusion, let $r \in \cl(\alpha \cdot \recc B + C)$, that is, $r = \lim_{n \rightarrow \infty} (\alpha r^{(n)}+c^{(n)})$ for some $(r^{(n)})_n \in \recc B$, $(c^{(n)})_n \in C$; let $b\in \cl(\alpha \cdot B+C)$, that is, $b = \lim_{n \rightarrow \infty} (\alpha b^{(n)}+\tilde{c}^{(n)})$ for some $(b^{(n)})_n \in B$, $(\tilde{c}^{(n)})_n \in C$. For any $\gamma \geq 0$, we have $b + \alpha r = \lim_{n \rightarrow \infty}(\alpha (b^{(n)}+\gamma r^{(n)}) + \tilde{c}^{(n)}+\gamma c^{(n)}) \in \cl (\alpha \cdot B + C)$. 	
	
\end{proof}

\begin{remark} \label{rem:selfbddintersection}
Note that $\mathcal{B}(\R^q,C)$ is closed under intersections. Consider a collection $(B_{\alpha})_{\alpha \in A} \in \mathcal{B}(\R^q,C)$ and let $b_{\alpha} \in \R^q$ be such that $B_{\alpha} \subseteq b_{\alpha} +\recc B_{\alpha}$, for $\alpha \in A$. The assertion holds trivially if the intersection is empty. Assume $\bigcap_{\alpha \in A} B_{\alpha} \neq \emptyset$. Then, $\recc (\bigcap_{\alpha \in A} B_{\alpha})\supseteq \bigcap_{\alpha \in A} \recc B_{\alpha} \supseteq C$. Let $b\in \bigcap_{\beta\in A}(b_{\beta}-\recc (\bigcap_{\alpha \in A}B_{\alpha}))$. Note that the existence of such $b$ is guaranteed as $\recc (\bigcap_{\alpha \in A}B_{\alpha})$ is solid and $\bigcap_{\alpha \in A} B_{\alpha} \neq \emptyset$. Then, it can be shown that $\bigcap_{\alpha \in A}B_{\alpha} \subseteq b + \recc(\bigcap_{\alpha \in A}B_{\alpha})$. 

Note that $(\mathcal{B}(\R^q,C),\oplus,\odot,\supseteq)$ would be a complete lattice with $$\sup (B_{\alpha})_{\alpha \in A} := \bigcap_{\alpha \in A} B_{\alpha}, \:\:\: \inf (B_{\alpha})_{\alpha \in A} := \cl \conv \bigcup_{\alpha \in A} B_{\alpha}$$ if both sets are in $\mathcal{B}(\R^q,C)$. Clearly, $\sup (B_{\alpha})_{\alpha \in A} \in \mathcal{B}(\R^q,C)$. However, $\cl \conv \bigcup_{\alpha \in {A}} B_{\alpha}$ is not necessarily self-bounded. Consider for instance, $B_x = (x,\; x^2)^T + \R^2_+ \in \mathcal{B}(\R^2,\R^2_+)$ for $x \in \R$. Note that $\cl \conv \bigcup_{x \in \R} B_{x} = \epi x^2 + \R^2_+ \notin \mathcal{B}(\R^q,\R^q_+)$ as $\recc (\epi x^2 + \R^2_+) = \R^2_+$, see also Remark~\ref{rem:notsolid}. 
\end{remark}

The following lemma together with Propositions~\ref{prop:upperset2} and~\ref{prop:upperset3} shows the importance of the concept of self-boundedness in terms of approximations of convex upper closed sets via convex sets of the form $\conv\{a^1,\ldots,a^s\}+K$. 
\begin{lemma}\label{lemma:weierstrass}
Let $A\subseteq \R^q$ be a compact and convex set, $K\subseteq \R^q$ be a non-trivial solid convex cone and $c\in \Int K$ be fixed. For any $\epsilon > 0$, there exists a finite set $\bar{A}\subseteq A$ such that $\conv \bar{A} + K -\epsilon \{c\} \supseteq A + K$. 
\end{lemma}

\begin{proof}
Let $B := {B}(0,2\epsilon \norm{c} )\cap K$, where $ \norm{\cdot}$ is the Euclidean norm. Define $$A_{\epsilon} := \conv \left[ ( A-\epsilon \{c\} ) \cup A \right] + B. $$ Note that $A \subseteq \Int A_{\epsilon}$. Indeed, for any $a \in A$, $a- \epsilon c \in  \conv \left[ (A-\epsilon \{c\} ) \cup A \right]$, and $\epsilon c \in \Int B$ as $c\in\Int K$. Furthermore, we have $A_{\epsilon} \subseteq A - \epsilon\{c\}+K$. To see, let $a \in A_{\epsilon}$. Note that $a = \sum_{i\in I}\alpha_i(a^i-\epsilon c)+\sum_{i\in J}\alpha_ia^i + b$, for some $N \in \N$, partition $I,J$ of $\{1,\ldots,N\}$, $a^i \in A$ for $i \in \{1,\ldots,N\}$, $b \in B$, and $\alpha_i \in [0,1]$ with $\sum_{i=1}^N\alpha_i =1$. Now, $a = \sum_{i=1}^N\alpha_i a^i - \epsilon c + (1-\sum_{i \in I}\alpha_i)\epsilon c + b \in A - \epsilon\{c\}+K$, since $\sum_{i=1}^N\alpha_i a^i \in A$, and $(1-\sum_{i\in I}\alpha_i)\epsilon c + b \in K$.

Let $(S_{\alpha})_{\alpha\in I}$ be the collection of all finite subsets of $A_{\epsilon}$ with at least $q+1$ elements. Define $\tilde{S}_{\alpha}:= \Int \conv S_{\alpha}$. $(\tilde{S}_{\alpha})_{\alpha \in I}$ is an open cover for $A$, and there exists a finite subcover as $A$ is compact, that is, there exists $\tilde{s}\in\N\setminus\{0\}$ such that $A \subseteq \bigcup_{n=1}
^{\tilde{s}} \tilde{S}_{\alpha_n} \subseteq \bigcup_{n=1}
^{\tilde{s}} \conv {S}_{\alpha_n}$. Let $\bigcup_{n=1}^{\tilde{s}}S_{\alpha_n}=\{v^1,\ldots,v^s\}$. Clearly, $A \subseteq \conv\{v^1,\ldots, v^s\}$. As $v^n \in {A}_{\epsilon} \subseteq  A - \epsilon\{c\}+K$, there exists $a^n\in A, k^n \in K$ such that $v^n = a^n -\epsilon c +k^n$ for all $n=1,\ldots,s$. Then, $\conv\{v^1,\ldots,v^s\} + K \supseteq A +K $ implies that $\bar{A}=\{a^1,\ldots,a^s\}$ satisfies the assertion.
\end{proof}

\begin{proposition}\label{prop:upperset2}
Let $A \in \mathcal{G}(\R^q,C)$ be bounded with respect to a non-trivial convex pointed cone $K\supseteq C$ and $c\in \Int C$ be fixed. Then, for any $\epsilon >0$, there exists a finite set of points $\bar{A}\subseteq A$ such that 
\begin{equation}\label{eq:innerouter}
\conv{\bar{A}}+K-\epsilon \{c\} \supseteq A.
\end{equation}
\end{proposition}
\begin{proof}
First, we show that there exists a compact set $B \subseteq A$ such that $B+K -\frac{\epsilon}{2}\{c\}\supseteq A$. Consider the sequence of sets given by $B_n := a + nc -C$ for $n \geq 1$, where $a \in A$ is fixed. Note that $B_{n} \subseteq B_{n+1}$ holds for all $n\geq 1$, and $\bigcup_{n\geq 1}B_n = \R^q$ as $C$ is solid. 
Hence, $\bigcup_{n\geq1}(B_n\cap A) + K \supseteq A$. Note that since $A$ is bounded with respect to $K$, there exists $p \in \R^q$ with $p \leq_{K} \tilde{a}$ for all $\tilde{a}\in B_n\cap A$. Moreover, $\tilde{a} \leq_K a +(n+1)c$ for all $\tilde{a}\in B_n\cap A$. Since both $B_n$ and $A$ are closed and $K$ is pointed, $B_n\cap A$ is compact for all $n\geq 1$. {Let $\epsilon_n := \inf\{\delta >0 \;\st (B_n\cap A)+K-\delta \{c\}\supseteq A\}$. As $A$ is bounded with respect to $K$, $\epsilon_n \in \R$ for all $n \geq 1$. Moreover, $(\epsilon_n)_{n\geq 1}$ is decreasing and $\lim_{n\rightarrow\infty}\epsilon_n = 0$ since $(B_n\cap A)+K \subseteq (B_{n+1}\cap A)+K$, and $\cup_{n\geq 1}(B_n\cap A)+K \supseteq A$.} Then, there exists $N>0$ such that $\epsilon_n < \frac{\epsilon}{2}$ for $n > N$ and $B = (B_{N+1}\cap A)$ satisfies the required property.\\
By Lemma~\ref{lemma:weierstrass}, there exists a finite set $\bar{A}\in B$ such that $\conv{\bar{A}}+K-\frac{\epsilon}{2} \{c\} \supseteq B+K$. Then, we have $\conv{\bar{A}}+K-\epsilon \{c\} \supseteq B+K - \frac{\epsilon}{2} \{c\}\supseteq A$. 
\end{proof}
\begin{remark}\label{rem:selfbdd}
Note that for $A\in \mathcal{G}(\R^q,C)$, $\recc A$ is a non-trivial convex cone and $\recc A \supseteq C$. Moreover, if $A$ is self-bounded, then by Proposition~\ref{prop:upperset2}, it is possible to generate finite outer approximation to $A$ using $\recc A$. Indeed, for any $\epsilon>0$ there exists a finite subset $\bar{A}$ of $A$ such that $A^{\text{out}}:=\conv{\bar{A}}+\recc A-\epsilon \{c\} \supseteq A$. Moreover, $\bar{A}$ also generates an inner approximation as $A^{\text{in}}:=\conv \bar{A} +\recc A \subseteq A$. {It is clear that the Hausdorf distance between the inner and the outer approximations is bounded}, namely, $h(A^{\text{out}},A^{\text{in}}) \leq \epsilon \norm{c}$.
\end{remark}

The following proposition shows that if $A$ is not self-bounded as in Definition~\ref{defn:bdd}, then it is not possible to find a polyhedral outer approximation $A^{\text{out}}\supseteq A$ such that the Hausdorff distance between $A^{\text{out}}$ and $A$ is finite. 

\begin{proposition}\label{prop:upperset3}
Let $A \in \mathcal{G}(\R^q,C)$ be not self-bounded but bounded with respect to $K$ for some non-trivial closed convex cone $K$. Let a finite set $\bar{Y} \subseteq \R^q$ satisfy $\conv \bar{Y} +K \supseteq A$. Then, $h(\conv \bar{Y}+K, A) = \infty$.
\end{proposition}

\begin{proof}
Since $A$ is not self-bounded but bounded with respect to $K$, $K \supsetneq \recc A$. Then, there exists $\bar{k}\in K \setminus \recc A$. For any $\bar{y}\in \bar{Y}$, there exists $M\geq 0$ such that $\{\bar{y}+\lambda \bar{k} \st \lambda \geq M \} \cap A = \emptyset$ as $A$ is convex and $\bar{k}$ is not a recession direction. Let $\bar{a}\in A$ be such that $\norm{\bar{a}-\bar{y}-M\bar{k}} = d(\bar{y}+M\bar{k}, A)$. As $\bar{k} \notin \recc A$, there exists $\alpha>0$ such that $\bar{a}+\alpha \bar{k} \notin A$. Then, there exists $\gamma \in \R^q\setminus \{0\}$ such that $\gamma^T\bar{a} + \alpha \gamma^T\bar{k} > \sup_{a\in A}\gamma^Ta$. Clearly, $\gamma^T\bar{k} > 0$. Let $H = \{y \in \R^q \st \gamma^Ty = \gamma^T\bar{a}+ \alpha \gamma^T \bar{k}\}$ and $\bar{H} = \{y \in \R^q \st \gamma^Ty \leq \gamma^T\bar{a}+ \alpha \gamma^T \bar{k}\} \supseteq A$. 
Consider $y^n := \bar{y}+ (M+n) \bar{k}$. On the one hand, as $\gamma^T\bar{k} > 0$, there exists $N \geq 1$ such that $y^n \notin \bar{H}$ for $n \geq N$. Let $d_n:=d(y^n, \bar{H})$. Then, for $n \geq N$, $d_n \leq d(y^n,A) \leq h(\conv \bar{Y} +K,A)$ as $y^n \in \conv \bar{Y}+K$. On the other hand, $y^n$ can be written as $y^n = y + d_n \frac{\gamma}{\norm{\gamma}}$ for some $y \in H$ as $\frac{\gamma}{\norm{\gamma}}$ is the unit normal vector to $H$. Then, for $n > N$, we have
\begin{align*}
d_n &= \frac{1}{\norm{\gamma}}(\gamma^Ty^n-\gamma^Ty)\\
&= \frac{1}{\norm{\gamma}} (\gamma^T y^N + (n-N)\gamma^T\bar{k} - \gamma^T\bar{a}- \alpha\gamma^T\bar{k}) \\
& > \frac{1}{\norm{\gamma}} (\gamma^T\bar{a}+\alpha\gamma^T\bar{k} + (n-N)\gamma^T\bar{k} - \gamma^T\bar{a}+ \alpha\gamma^T\bar{k}) \\
&= \frac{(n-N)\gamma^T\bar{k}}{\norm{\gamma}}
\end{align*}
Then, $h(\conv \bar{Y} +K,A) \geq \lim_{n \rightarrow \infty}d_n = \infty$.
\end{proof}

\section{Convex Vector Optimization}
\label{sect:CVOP}

\subsection{Problem Setting and Solution Concepts}
\label{subsect:Problemsetting}

A convex vector optimization problem (CVOP) with ordering cone $C$ is to
\begin{align*} \label{(P)}
 \text{minimize~} f(x) \text{~with respect to~} \leq_C \text{~subject to~}  g(x) \leq_D 0, \tag{P}
\end{align*}
where $C\subseteq\R^q$, and $D\subseteq\R^m$ are non-trivial pointed convex ordering cones with nonempty interior, $X\subseteq \R^n$ is a convex set, the vector-valued objective function $f: X \rightarrow \R^q$ is $C$-convex, and the constraint function $g: X \rightarrow \R^m$ is $D$-convex (see e.g., \cite{luc}). Note that the feasible set $\mathcal{X}:=\{x\in X: g(x)\leq_D 0 \} \subseteq X \subseteq\R^n$ of \eqref{(P)} is convex. Throughout we assume that \eqref{(P)}  is feasible, i.e., $\mathcal{X}\neq \emptyset$. The image of the feasible set is defined as $f(\mathcal{X}) = \{ f(x) \in \R^q : x \in \mathcal{X}\}$. The set 
\begin{equation}\label{upim1}
 \mathcal{P} := \cl(f(\mathcal{X})+C)	
\end{equation}
is called the \emph{upper image} of \eqref{(P)} (or {\em upper closed extended image of \eqref{(P)}}, see \cite{geometric}). Clearly, $\mathcal{P}$ is convex and closed. Hence $\mathcal{P}\in \mathcal{G}(\R^q,C)$. Moreover, $\bd \mathcal{P} \cap f(\mathcal{X})=\wMinc f(\mathcal{X})$, see, for instance, Proposition 4.1 in \cite{ehrgott}. 
\begin{remark}
	Proposition~4.1 in~\cite{ehrgott} also states that $\mathcal{P}$ is bounded with respect to $\R^q_+$ but this is true only under the assumption that the feasible region $\mathcal{X}$ is bounded, see Example~\ref{ex:expon} for a simple counterexample. 
\end{remark}
 
\begin{definition}\label{bdd}
Let $K$ be a closed convex cone such that $K \supseteq C$. Problem \eqref{(P)} is said to be \emph{bounded with respect to $K$} if $\mathcal{P}\in\mathcal{G}(\R^q,C)$ is bounded with respect to $K$. \eqref{(P)} is said to be \emph{bounded} if it is bounded with respect to $C$ and \emph{unbounded} if it is not bounded.
\end{definition}

There are different solution concepts regarding CVOPs. The following 
is a well-known solution concept which is also known as a (weakly) efficient solution or (weakly) Pareto optimal solution of~\eqref{(P)}. 
\begin{definition} \label{defn:minimizer}
An element $\bar{x}$ of $\mathcal{X}$ is said to be a \emph{minimizer} if $f(\bar{x}) \in \Minc f(\mathcal{X})$ and a \emph{weak minimizer} if $f(\bar{x}) \in \wMinc f(\mathcal{X})$.
\end{definition} 

Note that the image of a (weak) minimizer is a single point on the boundary of the upper image. For bounded convex vector optimization problems, an $\epsilon$-solution concept which generates inner and outer approximations to the whole upper image is given in~\cite{cvop} as follows.
\begin{definition}\cite[Definition 3.3]{cvop}\label{defn:weaksoln}
For a bounded problem \eqref{(P)}, a nonempty finite set $\mathcal{\bar{X}}$ of (weak) minimizers is called a \emph{finite (weak) $\epsilon$-solution of \eqref{(P)}} if 
\begin{equation} \label{eqweaksoln}
\conv f(\mathcal{\bar{X}})+C-\epsilon \{c\} \supseteq \mathcal{P}.
\end{equation}
\end{definition}

Clearly, this definition suggests polyhedral inner and outer approximations to the upper image as follows $$\conv f(\mathcal{\bar{X}})+C-\epsilon \{c\} \supseteq \mathcal{P} \supseteq \conv f(\mathcal{\bar{X}})+C.$$ 

Note that if problem~\eqref{(P)} is bounded with respect to $K$ for some non-trivial closed convex cone $K \supseteq C$ and moreover if $K$ satisfies $\mathcal{P}+K \subseteq \mathcal{P}$, then the problem, where cone $K$ is taken as the ordering cone, is equivalent to the original problem. In other words, $\cl (f(\mathcal{X})+C) = \cl (f(\mathcal{X})+K)$ and hence $\wMinc f(\mathcal{X}) = \wMinK f(\mathcal{X})$. If such $K$ exists, then it has to be the recession cone $\recc \mathcal{P}$ of the upper image by Proposition~\ref{prop:upperset1}. In the following definition, we suggest that we call a problem self-bounded if such cone $K$ exists.  

\begin{definition}\label{defn:selfbddP}
\eqref{(P)} is said to be \emph{self-bounded} if $\mathcal{P}\neq \R^q$ and if~\eqref{(P)} is bounded with respect to $\recc\mathcal{P}$. 
\end{definition}

By Proposition~\ref{prop:upperset3}, it is known that if a problem is not self-bounded then for any polyhedral outer approximation to the upper image, the Hausdorff distance between the outer approximation and the upper image is not finite. In particular, there exists no finite weak $\epsilon$- solution of~\eqref{(P)}. The following is a trivial example of a not self-bounded CVOP. 

\begin{ex}\label{ex:expon}
	Consider the biobjective optimization problem where the ordering cone is the positive orthant $\R^2_+$, the two objective functions to be minimized are $f_1(x) = x$, $f_2(x)= e^{-x}$ and the feasible region is the real line, $\mathcal{X}=\R$. Clearly, the image of the feasible region is the graph of $f_2$, namely $f(\mathcal{X})=\{(x,y)\in\R^2 \st y = e^{-x}\}$ and the upper image is the epigraph of $f_2$, $\mathcal{P}=\{(x,y)\in\R^2 \st y \geq e^{-x}\}$. Note that the problem is not bounded since for any $(x,y) \in \R^2$, we have $(x,y) + \R^2_+ \nsupseteq \mathcal{P}$. Moreover, one can easily check that the recession cone of $\mathcal{P}$ is $\R^2_+$. Hence, this problem is not self-bounded and it is not possible to find a polyhedral outer approximation to the upper image for a given error bound. 
\end{ex}

As seen in the above example, there are convex vector optimization problems that are not \emph{tractable} in terms of having polyhedral outer approximations and clearly problems that are not self-bounded are not tractable in that sense. The following example provides a non-linear convex tractable problem with a non-compact feasible region.

\begin{ex}
Consider the biobjective optimization problem with a solid ordering cone $C \subsetneq \R^2_+$, where the two objective functions to be minimized are $f_1(x) = x$, $f_2(x)= x^{-1}$ and the feasible region is the positive real line. Clearly, $\mathcal{P} = \epi f_2 \cap \Int \R^2_+$ and $\recc \mathcal{P}= \R^2_+$. The problem is not bounded as the ordering cone is strictly smaller than the positive orthant. However, it is self-bounded, hence tractable in the sense of polyhedral approximations. 
\end{ex}

Assume for now that problem~\eqref{(P)} is self-bounded. Clearly, if one can compute the recession cone~$\recc \mathcal{P}$ of the upper image~$\mathcal{P}$ and if the recession cone is polyhedral and pointed, then it is possible to apply the approximation algorithms for bounded CVOPs in the literature, where the ordering cone is taken as $\recc\mathcal{P}$, see \cite{ehrgott,cvop}. 
{Note that, it is in general difficult to check if the recession cone of the upper image is polyhedral and pointed. However, it is trivially polyhedral if the objective function is $\R^2$-valued for instance. In this case, it would be either pointed or a halfspace. Note that if the recession cone of the upper image of a two-dimensional CVOP is a halfspace, then the upper image itself is a halfspace by convexity. Then one could simplify the problem to a linear vector optimization problem and apply for instance, the parametric simplex algorithm from \cite{lvopalg}, which works even if the upper image is a halfspace.}

In each iteration of algorithms provided in~\cite{ehrgott,cvop}, a scalarized problem is solved. In particular, the \emph{weighted sum scalarization} of~\eqref{(P)}, which is given by
\begin{align} \label{P1}  \tag{P$_w$}
\min \left\{w^Tf(x) \st x \in \mathcal{X} \right\},  
\end{align}
for $w \in \R^q$, is solved in each iteration of the geometric dual algorithm. It is also solved at the initialization step of the 'primal' Benson's algorithm. The followings are well-known results, see e.g., \cite{jahn,luc}. 
\begin{proposition} \label{dualprop1}
	Let $w \in C^+\setminus\{0\}$. An optimal solution $x^w$ of \eqref{P1} is a weak minimizer of \eqref{(P)}.
\end{proposition}

\begin{theorem}[{\cite[Corollory 5.3]{jahn}}]
	\label{thm:scalarizations}
	If $\mathcal{X}\subset \R^n$ is a non-empty closed set and~\eqref{(P)} is a convex problem, then for each weak minimizer $\bar{x}$ of~\eqref{(P)}, there exists $w \in C^+\setminus \{0\}$ such that $\bar{x}$ is an optimal solution to~\eqref{P1}.
\end{theorem}

\subsection{Self-Bounded Problems}	\label{subsect:selfbdd}
In this section, we consider self-bounded problems and provide a characterization result in terms of weighted sum scalarization problems. 

Recall that a vector optimization problem is said to be linear if the objective function is $f(x)=Px$ where $P \in \R^{q\times n}$, the feasible region $\mathcal{X}$ and the ordering cone are polyhedral. Note that LVOPs are self-bounded as long as the upper image $\mathcal{P}$ is not the whole space. Indeed, for a LVOP with $\emptyset\neq\mathcal{P}\neq\R^q$, the recession cone of the upper image $\recc \mathcal{P}$ is polyhedral, it can be computed by solving the so called \emph{homogeneous problem}, and problem~\eqref{(P)} is bounded with respect to $\recc \mathcal{P}$, see, for instance, \cite{lohne} for the details. Moreover, it is also known that for linear problems we have $(\recc \mathcal{P})^+ = \{w \in C^+ \st \eqref{P1} \text{~is bounded}\}$, see~\cite{lvopalg}.

In order to give a characterization of the self-bounded convex vector optimization problems we define 
\begin{equation}\label{eq:B}
W := \{w \in C^+ \st \eqref{P1} \text{~is bounded}\}.
\end{equation} 
\begin{remark}\label{rem:Wcone}
It is easy to show that $W$ is a convex cone. Note also that $W$ is not necessarily a closed cone. Consider Example~\ref{ex:expon}. Note that $w = (r, \:\: 0)^T \in \R^2_+$ makes the sum scalarization problem~\eqref{P1} unbounded for any $r>0$. On the other hand, for any other $w \in \R^2_+$, \eqref{P1} is bounded.	Hence, $W = \Int \R^2_+ \cup \{(0, \: r)^T \st r\geq 0\}$, and this is not a closed set.
\end{remark}

\begin{remark}\label{rem:barrier}
	Note that ~\eqref{P1} can be reformulated as $\min \left\{w^Ty \st y \in f(\mathcal{X})+C \right\}$. Then, $W$ is the negative of the barrier cone of the upper image, namely $W = -b(\mathcal{P})$, where \begin{align*}
	b(\mathcal{P}):= \{w\in \R^q \st \sup_{y\in\mathcal{P}}w^Ty < +\infty\}.
	\end{align*} 
	It is known that for any nonempty closed convex set $A \in \R^q$, $$\cl b(A) = (\recc A)^- = \{w\in \R^q \st w^Ty \leq 0 \text{~for all~} y \in \recc A\},$$ see, for instance,~\cite{adly2004,zalinescu}. Moreover, if $A$ is \emph{hyperbolic}, that is, if $A \subseteq D + \recc A$ for some bounded $D \subseteq \R^q$, then $b(A)$ is closed, see \cite[Proposition 1.1]{goossens1986}. 
\end{remark}

The following results relates the recession cone of the upper image and $W$ for convex vector optimization problems. Indeed, it is seen that the above-mentioned result for LVOPs holds also for the self-bounded CVOPs.

\begin{proposition} \label{prop:BinPinf}
It is true that $(\recc \mathcal{P})^+ = \cl W$. Moreover, if problem~\eqref{(P)} is self-bounded, then $(\recc\mathcal{P})^+ = W$.  
\end{proposition}
\begin{proof}
The statements follow by Remark~\ref{rem:barrier}. 
\end{proof}

It is clear that self-boundedness of problem~\eqref{(P)} also guarantees that $W$ is a closed set. Next, we show that $W = (\recc\mathcal{P})^+$ implies that problem \eqref{(P)} is self-bounded as long as $\mathcal{P}\neq \R^q$. 

For the following lemma and the theorem, consider a basis of $(\recc \mathcal{P})^+$ given by 
\begin{equation}\label{eq:Lambda}
\Lambda := \{w \in (\recc\mathcal{P})^+ \st w^Tc = 1\},
\end{equation} where $c \in \Int C$ is fixed. Note that $\Lambda$ is a compact set. Moreover, as $\recc \mathcal{P} \supseteq C$, we have $(\recc \mathcal{P})^+\subseteq C^+$. Then, $\Lambda = \emptyset$ holds if only if $(\recc \mathcal{P})^+=\{0\}$, hence $\recc \mathcal{P} = \R^q$ and $\mathcal{P}= \R^q.$

\begin{lemma}\label{lemma:Lambda}
Assume $\{0\} \neq (\recc\mathcal{P})^+ \subseteq W$. Then $$\mathcal{P}^0 := \bigcap_{w \in \Lambda}\{y\in \R^q \st w^Ty \leq \inf_{x\in \mathcal{X}}w^Tf(x)\}\neq \emptyset.$$
\end{lemma}

\begin{proof}
$\{0\} \neq (\recc\mathcal{P})^+$ implies that $\Lambda \neq \emptyset.$ As $\Lambda \subseteq (\recc\mathcal{P})^+ \subseteq W$, it is true that $\gamma^w:=\inf_{x\in \mathcal{X}}w^Tf(x) > -\infty$ for $w \in \Lambda$. Moreover, as $\recc\mathcal{P} \supseteq C$, it is true that $\Lambda \subseteq (\recc\mathcal{P})^+ \subseteq C^+$. Hence, $w^Tc >0$ for any $w \in \Lambda$.\\
For contradiction, assume $\mathcal{P}^0=\emptyset.$ Then, for all $y \in \R^q$, there exists $w \in \Lambda$ such that $w^Ty > \gamma^w$. In particular, consider $y=-nc \in \R^q$ for $n \geq 1.$ Then, there exists $w_n \in \Lambda$ with $-nw_n^Tc > \inf_{x\in \mathcal{X}}w_n^Tf(x) = \gamma^{w_n}$ for $n\geq 1$. Note that $\Lambda \subseteq \R^q$ is compact and by Bolzano Weierstrass Theorem, it is sequentially compact. That is, there exists a convergent subsequence $(w_{n_k})_{k\geq 1}$ of $(w_n)$ with $\lim_{k\rightarrow \infty}w_{n_k} = w \in \Lambda$. Then, 
\begin{align*}
\lim_{k\rightarrow \infty}\inf_{x\in\mathcal{X}}w_{n_k}^Tf(x) \leq \lim_{k\rightarrow\infty}(-n_kw_{n_k}^Tc) = -\infty
\end{align*}
as $w_{n_k}^Tc >0$ and $\lim_{k\rightarrow\infty} n_k = \infty$.

Since $f$ is $C$-convex and $w_{n_k} \in W$, $(w_{n_k}^Tf)_{k\geq 1}$ are finite convex functions on $\R^n$. Moreover, $w_{n_k}^Tf$ converges point-wise to $w^Tf$. By Theorem 10.8 of \cite{rockafellar}, $(w_{n_k}^Tf)_{k\geq 1}$ converges uniformly to $w^Tf$ on each compact subset of $\R^n$. Let $\mathcal{X}_m := \mathcal{X}\cap B(0,m)$. Clearly, $\mathcal{X}_m$ is compact for all $m\geq 1$ and $\bigcup_{m\geq 1}\mathcal{X}_m = \mathcal{X}$. Since $w_{n_k}^Tf$ uniformly converges to $w^Tf$ on $\mathcal{X}_m$, we have $$\inf_{x\in\mathcal{X}_m}w^Tf(x)=\inf_{x\in\mathcal{X}_m}\lim_{k\rightarrow\infty}w_{n_k}^Tf(x)= \lim_{k\rightarrow \infty}\inf_{x\in\mathcal{X}_m}w_{n_k}^Tf(x)=:b_m.$$ Moreover, as $b_m$ is a decreasing sequence in $\R$, $\lim_{m\rightarrow \infty}b_m$ exists and {$$\lim_{m\rightarrow\infty}b_m=\inf_{x\in\mathcal{X}}w^Tf(x)=\lim_{k\rightarrow \infty}\inf_{x\in\mathcal{X}}w_{n_k}^Tf(x) = -\infty,$$} which contradicts to the fact that $w \in \Lambda \subseteq W$.
\end{proof} 

\begin{theorem}\label{thm:2}
If $\{0\}\neq(\recc\mathcal{P})^+ = W$, then problem~\eqref{(P)} is self-bounded. 
\end{theorem}

\begin{proof}
If $(\recc \mathcal{P})^+\neq\{0\}$, then $\recc\mathcal{P}\neq \R^q$ and $\mathcal{P}\neq \R^q$. Moreover, as $\{0\}\neq (\recc\mathcal{P})^+ = W$, we have $\mathcal{P}^0 \neq \emptyset$ by Lemma~\ref{lemma:Lambda}. Let $\bar{y} \in \mathcal{P}^0$. Below, we show that $\{\bar{y}\} + \recc\mathcal{P} \supseteq \mathcal{P}$, hence,~\eqref{(P)} is self-bounded. Assume the contrary, that is, there exists $\bar{x}\in\mathcal{X}$ with $f(\bar{x})\notin \{\bar{y}\}+\recc\mathcal{P}$. Using the separation argument, there exists $\bar{w} \in \R^q\setminus\{0\}$ such that $\bar{w}^Tf(\bar{x}) < \bar{w}^T\bar{y} + \inf_{p\in\recc\mathcal{P}}\bar{w}^Tp$. Then, $\bar{w} \in (\recc\mathcal{P})^+$ and $\inf_{p\in\recc\mathcal{P}}\bar{w}^Tp \geq 0$. Let $\tilde{w}:=\frac{\bar{w}}{\bar{w}^Tc}$. Clearly, $\tilde{w}\in\Lambda$ and $\tilde{w}^T\bar{y} > \tilde{w}^Tf(\bar{x})\geq \inf_{x\in\mathcal{X}}\tilde{w}^Tf(x)$, which contradicts to the fact that $\bar{y}\in\mathcal{P}^0$. 
\end{proof}

\begin{remark}
	The relation between hyperbolic sets and their barrier cones is studied by Zaffaroni in~\cite{zaffaroni2013}. Indeed an easier proof of Theorem~\ref{thm:2} can be done using Theorem 6.5 in~\cite{zaffaroni2013}. The direct proof provided above uses the set $\mathcal{P}^0$, which in general is taken as the initial outer approximation to the upper image in Beson-type approximation algorithms, see~\cite{ehrgott_shao,cvop}. Hence, Lemma~\ref{lemma:Lambda} is also important for practical reasons.
\end{remark}

\section{Concluding Remarks} \label{sect:concl}
The aim of this study is to understand the structure of possibly unbounded convex vector optimization problems and to see if those are tractable in the sense that there are polyhedral inner and outer approximations of the upper image. It has been shown that not all CVOPs are tractable and indeed, only the ones that are self-bounded are tractable. 

For the problems which are not known to be bounded, no solution concept which generates inner and outer approximations to the upper image is known in the literature. Clearly, for problems that are not self-bounded, it is not possible to come up with such a concept as it is not possible to find a polyhedral outer approximation, see Proposition~\ref{prop:upperset3} and Definition~\ref{defn:selfbddP}. However, one could generalize the concept of (weak) $\epsilon$-solution given for bounded problems to the self-bounded problems, using the recession cone of the upper image as follows. 

\begin{definition}\label{defn:weaksolnselfbdd}
	For a self-bounded problem \eqref{(P)}, a nonempty finite set $\mathcal{\bar{X}}$ of (weak) minimizers is called a \emph{finite (weak) $\epsilon$-solution of \eqref{(P)}} if 
	\begin{equation*} \label{eqweaksolnselfbdd}
	\conv f(\mathcal{\bar{X}})+\recc \mathcal{P}-\epsilon \{c\} \supseteq \mathcal{P}.
	\end{equation*}
\end{definition}

This definition is valid in the sense that the existence is known, see Proposition~\ref{prop:upperset2} and Definition~\ref{defn:selfbddP}. However, since it is difficult to compute $\recc \mathcal{P}$ in general, it is not really practical. Note on the other hand that for two dimensional self-bounded CVOPs, $\recc \mathcal{P}$ is polyhedral and there are two extreme directions generating it as long as $\recc \mathcal{P}$ is not a halfspace. 

As discussed in Section~\ref{sect:intro}, Benson-type algorithms proposed in~\cite{ehrgott, cvop} are designed to solve bounded CVOPs. Note that a problem is bounded if and only if it is self-bounded and $\recc \mathcal{P}=C$. Then, by Proposition~\ref{prop:BinPinf} and Theorem~\ref{thm:2}, a problem is bounded if and only if $W = C^+$. As $W$ is known to be a convex cone, see Remark~\ref{rem:Wcone}, in order to check if a problem is bounded, it is enough to check if~\eqref{P1} is bounded for all extreme directions $w$ of $C^+$. 

Note that the algorithms in~\cite{ehrgott,cvop} can be extended to solve self-bounded problems as long as the recession cone of the upper image is polyhedral and its extreme directions can be computed, which remains as an open problem.

\section*{Acknowledgment}
We are grateful to the anonymous referees for insightful comments allowed us to correct some inaccuracies appearing in the preceding version and for numerous suggestions that improved the presentation. We would also like to thank Prof. Alberto Zaffaroni for his constructive discussion during the revision process.

\bibliographystyle{plain}
\bibliography{reportbib}

\end{document}